\documentclass[11pt]{article}
\usepackage[utf8]{inputenc}
\usepackage{amssymb,amsmath, amsthm, mathabx}
\usepackage{color, enumerate}
\usepackage{comment}
\usepackage{hyperref}
\usepackage[margin=0.9in]{geometry}
\usepackage{tikz}
\usetikzlibrary{arrows,calc}
\usepackage{graphicx}
\usepackage{epstopdf}
\usepackage[percent]{overpic}
\usepackage{accents}
\usepackage[font=small]{caption}

\DeclareMathOperator{\bbN}{\mathbb{N}}

\DeclareMathOperator{\bbP}{\mathbb{P}}

\DeclareMathOperator{\bbR}{\mathbb{R}}

\DeclareMathOperator{\bbZ}{\mathbb{Z}}

\DeclareMathOperator{\bfP}{\mathbf{P}}
\DeclareMathOperator{\bfQ}{\mathbf{Q}}

\DeclareMathOperator{\bfa}{\mathbf{a}}
\DeclareMathOperator{\bfb}{\mathbf{b}}

\DeclareMathOperator{\bfi}{\mathbf{i}}

\DeclareMathOperator{\sT}{\mathcal{T}}

\DeclareMathOperator{\one}{\mathbf{1}}
\DeclareMathOperator{\E}{E}

\DeclareMathOperator{\lc}{\lceil}
\DeclareMathOperator{\rc}{\rceil}
\DeclareMathOperator{\lf}{\lfloor}
\DeclareMathOperator{\rf}{\rfloor}

\newtheorem{thm}{Theorem}[section]
\newtheorem{prop}[thm]{Proposition}
\newtheorem{lem}[thm]{Lemma}
\newtheorem{cor}[thm]{Corollary}

\newcommand{\ubar}[1]{\underaccent{\bar}{#1}}

\allowdisplaybreaks
\numberwithin{equation}{section}
\numberwithin{figure}{section}

\graphicspath{ {./} }

\title{Limit shapes for inhomogeneous corner growth models with exponential and geometric weights}
\author{Elnur Emrah\thanks{Department of Mathematics, University of Wisconsin-Madison. Madison, WI. USA. E-mail: \href{mailto:emrah@math.wisc.edu}{emrah@math.wisc.edu}}}

\begin{document}

\maketitle 
\abstract{
We generalize the exactly solvable corner growth models by choosing the rate of the exponential distribution $a_i+b_j$ and the parameter of the geometric distribution $a_i b_j$ at site $(i, j)$, where $(a_i)_{i \ge 1}$ and $(b_j)_{j \ge 1}$ are jointly ergodic random sequences. We identify the shape function in terms of a simple variational problem, which can be solved explicitly in some special cases. 
}


\section{Introduction}

The corner growth model is a frequently studied model of random growth. See \cite{Seppalainen} for a detailed introduction, and \cite{Martin}, \cite{Seppalainen4} for an overview of related research. The model describes a cluster of sites that emerges from the corner of a quadrant and grows over time. We represent the quadrant with $\bbN^2$ and the cluster with a family of subsets $S_t \subset \bbN^2$ indexed by time $t \ge 0$. Each site $(i, j) \in \bbN^2$ waits for 
$(i-1, j)$ if $i > 1$ and $(i, j-1)$ if $j > 1$ to be in the cluster and, after an additional waiting time of $W(i, j)$, joins the cluster permanently at time $G(i, j)$. More precisely, $S_t = \{(i, j) \in \bbN^2: G(i, j) \le t\}$ for $t \ge 0$ and 
\begin{equation}
\label{e2}
G(i, j) = G(i-1, j) \vee G(i, j-1) + W(i, j) \quad \text{ for } i, j \in \bbN, 
\end{equation}
where $G(i, 0) = G(0, j) = 0$ for $i, j \in \bbN$. As first observed in \cite{Muth}, the preceding recursion implies
\begin{equation}\label{eq51}G(m, n) = \max_{\pi \in \Pi_{1, 1, m, n}} \sum \limits_{(i, j) \in \pi} W(i, j) \quad \text{ for } m, n \in \bbN, \end{equation}
where $\Pi_{u, v, u', v'}$ is the set of all directed paths from $(u, v)$ to $(u', v')$, that is, all finite sequences $\pi = ((i_k, j_k))_{1 \le k \le l}$ in $\bbZ^2$ such that $(i_{k+1}-i_k, j_{k+1}-j_k) \in \{(1, 0), (0, 1)\}$ for $1 \le k < l$, $(i_1, j_1) = (u, v)$ and $(i_l, j_l) = (u', v')$. We will refer to the random quantities $\{W(i, j): i, j \in \bbN\}$ and $\{G(i, j): i, j \in \bbN\}$ as \emph{weights} and \emph{last-passage times}, respectively, because of connection (\ref{eq51}) with directed last-passage percolation. The problem is typically to understand the statistical properties of the last-passage times given the joint probability distribution $P$ of the weights.

Consider the corner growth model in which the weights are i.i.d. and exponentially distributed with rate $\lambda > 0$ i.e. $P(W(i, j) \ge x) = e^{-\lambda x}$ for  $i, j \in \bbN$ and $x \ge 0$. We will refer to this special case as the \emph{exponential model}. In a seminal paper \cite{Rost}, H. Rost observed the equivalence of the exponential model with the totally asymmetric simple exclusion process (TASEP) and proved that  
\begin{align}
\label{E1}
\lim \limits_{n \rightarrow \infty} \frac{G(\lf ns \rf, \lf nt \rf)}{n} = \frac{(\sqrt{s}+\sqrt{t})^2}{\lambda} \quad \text{ for } s, t > 0 \quad P\text{-a.s.} 
\end{align}
He also interpreted (\ref{E1}) as an asymptotic shape result in the sense that 
the rescaled cluster $S_n/n$ converges as $n \rightarrow \infty$ to the parabolic region $\{s,t \in \bbR_+: \sqrt{s}+\sqrt{t} \le \sqrt{\lambda}\}$ in the Hausdorff metric. 
The discrete counterpart of the exponential model is the \emph{geometric model} in which the weights are i.i.d. and geometrically distributed with (fail) parameter $q \in (0, 1)$ i.e. $P(W(i, j) \ge k) = q^k$ for $i, j \in \bbN$ and $k \in \bbZ_+$.  For the geometric model, one can also compute that 
\begin{align}
\label{E2}
\lim \limits_{n \rightarrow \infty} \frac{G(\lf ns \rf, \lf nt \rf)}{n} = \frac{q}{1-q}(s+t) + \frac{2\sqrt{q}}{1-q} \sqrt{st} \quad \text{ for } s , t > 0 \quad P\text{-a.s., }
\end{align}
\cite{CohnElkiesPropp}, \cite{JockuschProppShor}, \cite{Seppalainen2}. The existence of the deterministic (nonrandom) a.s. limits in (\ref{E1}) and (\ref{E2}) follows from standard subadditive arguments, which hold in greater generality, for example, when $P$ is an arbitrary i.i.d. measure. Define the \emph{shape function} $g$ by $g(s, t) = \lim_{n \rightarrow \infty}n^{-1}G(\lf ns \rf, \lf nt \rf)$ for $s, t > 0$ $P$-a.s.  
The basic properties of $g$ such as concavity and homogeneity are also easily derived \cite[Theorem~2.1]{Seppalainen}. An interesting matter is the explicit identification of $g$. Despite much effort, this has not been possible except for the exponential and geometric models, which are called \emph{exactly solvable} cases. In this paper, we introduce certain corner growth models with non i.i.d $P$ whose shape functions exist and can be determined explicitly. 

We study \emph{inhomogeneous} generalizations of the exponential and geometric models in which the parameters $\lambda$ and $q$ are site-dependent and drawn randomly from an ergodic distribution. More specifically, let $\bfa = (a_n)_{n \in \bbN}$ and $\bfb = (b_n)_{n \in \bbN}$ be stationary random sequences in $(0, \infty)$ such that the distribution $\mu$ of $(\bfa, \bfb)$ is separately ergodic under the map $\tau_k \times \tau_l$ for each $k, l \in \bbN$, where $\tau_k$ is the shift $(c_n)_{n \in \bbN} \mapsto (c_{n+k})_{n \in \bbN}$ for $k \in \bbZ_+$. In particular, $\bfa$ and $\bfb$ can be independent i.i.d. sequences. Suppose that, given $(\bfa, \bfb)$, the weights are conditionally independent and $W(i, j)$ is exponentially distributed with rate $\lambda = a_i+b_j$ for $i, j \in \bbN$. The inhomogeneous geometric model is defined similarly except now $\bfa$ and $\bfb$ are sequences in $(0, 1)$ and, given $(\bfa, \bfb)$, $W(i, j)$ is geometrically distributed with parameter $q = a_i b_j$ for $i, j \in \bbN$. Let us write $\bbP$ for the joint distribution of the weights and $\bfP_{\bfa, \bfb}$ for the joint conditional distribution of the weights given $(\bfa, \bfb)$. To be clear, $\bfP_{\bfa, \bfb}$ is a product measure on $\bbR_+^{\bbN^2}$ and $\bbP(B) = \int \bfP_{\bfa, \bfb}(B) \mu(d\bfa, d\bfb)$ for any Borel set $B \subset \bbR_+^{\bbN^2}$. Identifying $W(i, j)$ with the projection $\bbR_+^{\bbN^2} \rightarrow \bbR_+$ onto coordinate $(i, j)$,
\begin{align*}
\bfP_{\bfa, \bfb}(W(i, j) \ge x) = e^{-(a_i+b_j)x} \quad \text{ for } i, j \in \bbN \text{ and } x \ge 0 
\end{align*} 
for the exponential model and 
\begin{align*}
\bfP_{\bfa, \bfb}(W(i, j) \ge k) = a_i^k b_j^k \quad \text{ for } i, j \in \bbN \text{ and } k \in \bbZ_+
\end{align*} for the geometric model. A noteworthy feature that distinguishes these models from the classical homogeneous counterparts is the correlations of the weights along the rows and the columns, that is, $W(i, j)$ and $W(i', j')$ are not independent under $\bbP$ if $i = i'$ or $j = j'$.

Our main result is a simple variational description of the shape function. 
To state it, let $\alpha$ and $\beta$ denote the distributions of $a = a_1$ and $b = b_1$, respectively. For the exponential model, 
\begin{align}
g(s, t) = \inf \limits_{z \in [-\ubar{\alpha}, \ubar{\beta}]} \left\{s \int_0^\infty \frac{\alpha(dx)}{x+z} + t \int_0^\infty \frac{\beta(dx)}{x-z}\right\} \quad \text{ for } s, t > 0, \label{E6}
\end{align}
where $\ubar{\eta}$ is the left endpoint of the support (the complement of the largest open $\eta$-null set) of a Borel measure $\eta$ on $\bbR$. (We will also use $\bar{\eta}$ for the right endpoint of the support). When $\alpha$ and $\beta$ are uniform measures, we can compute $g$ explicitly. For example, if the supports of $\alpha$ and $\beta$ are the interval $[1/2, 3/2]$ then, for $s, t > 0$,    
\begin{align}
g(s, t) = s \log\bigg(2 + \frac{t-s+\sqrt{(t-s)^2+16st}}{4s}\bigg) + t \log\bigg(2 + \frac{s-t+\sqrt{(s-t)^2+16st}}{4t}\bigg).
\label{E7}
\end{align}
We obtain similar results for the geometric model in which explicit formulas arise when $\alpha$ and $\beta$ have densities proportional to $x \mapsto 1/x$. We deduce from (\ref{E6}) that $g$ is linear or infinite if $\ubar{\alpha}=\ubar{\beta} = 0$, and is differentiable if $\ubar{\alpha}+\ubar{\beta} > 0$. In the latter case, in contrast with (\ref{E1}), $g$ may be linear close to the axes depending on the behaviors of $\alpha$ and $\beta$ near $\ubar{\alpha}$ and $\ubar{\beta}$. More precisely, there exist constants $0 \le c_1 < c_2 \le \infty$ such that $g$ is strictly concave only inside the cone $c_ 1 < s/t < c_2$, see Corollary \ref{c3}. This is illustrated through the level set $g = 1$ in Figure \ref{F1} below. We expect the finer statistics of $G(\lf ns \rf, \lf nt \rf)$ as $n \rightarrow \infty$ to be qualitatively different in the linear and concave sectors. This has been confirmed for large deviation properties in \cite{EmrahJanjigian}. A project currently underway is to understand the limit fluctuations. 


\begin{figure}[h!]
\centering
\includegraphics[scale=0.6]{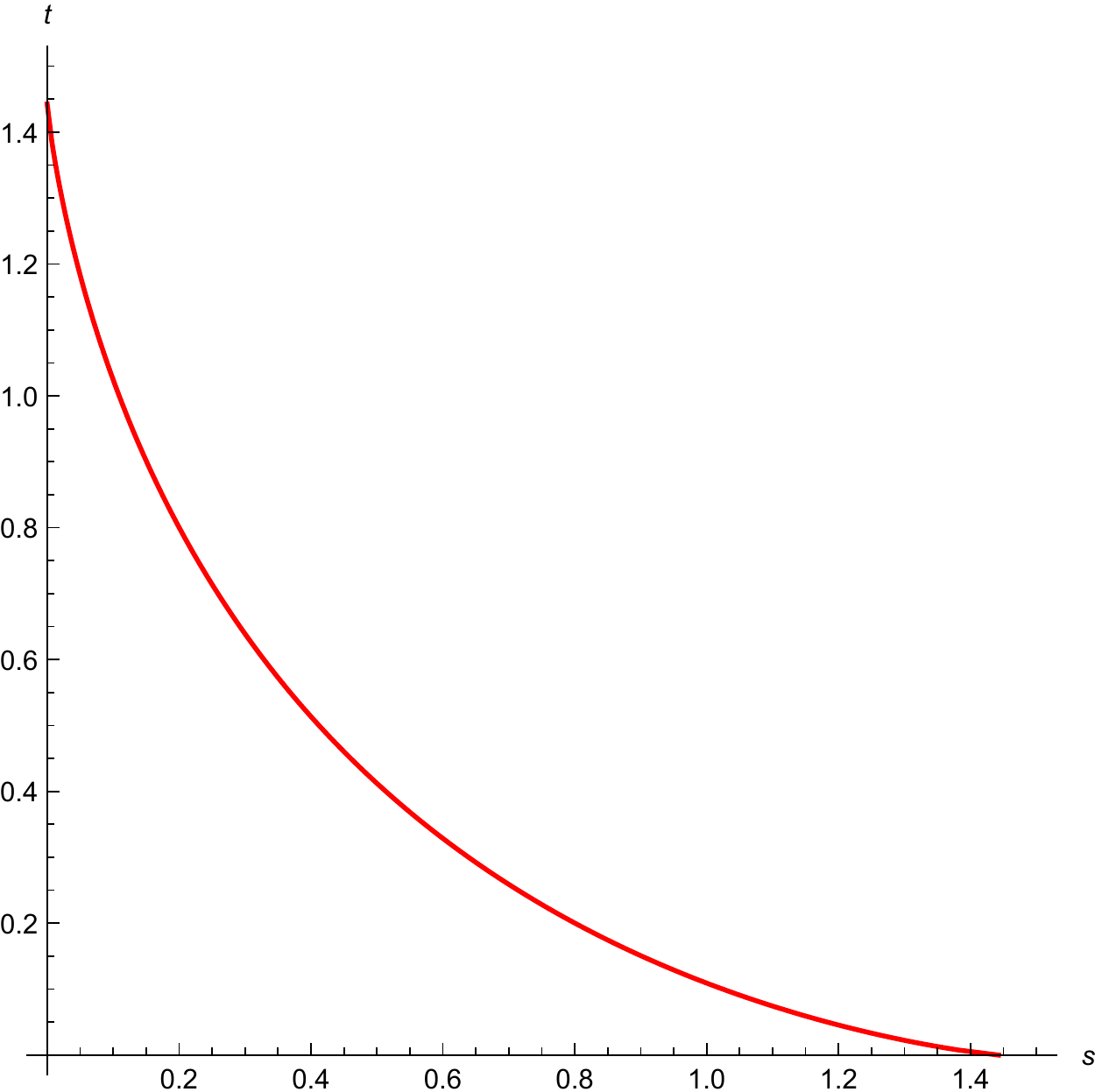}
\includegraphics[scale=0.6]{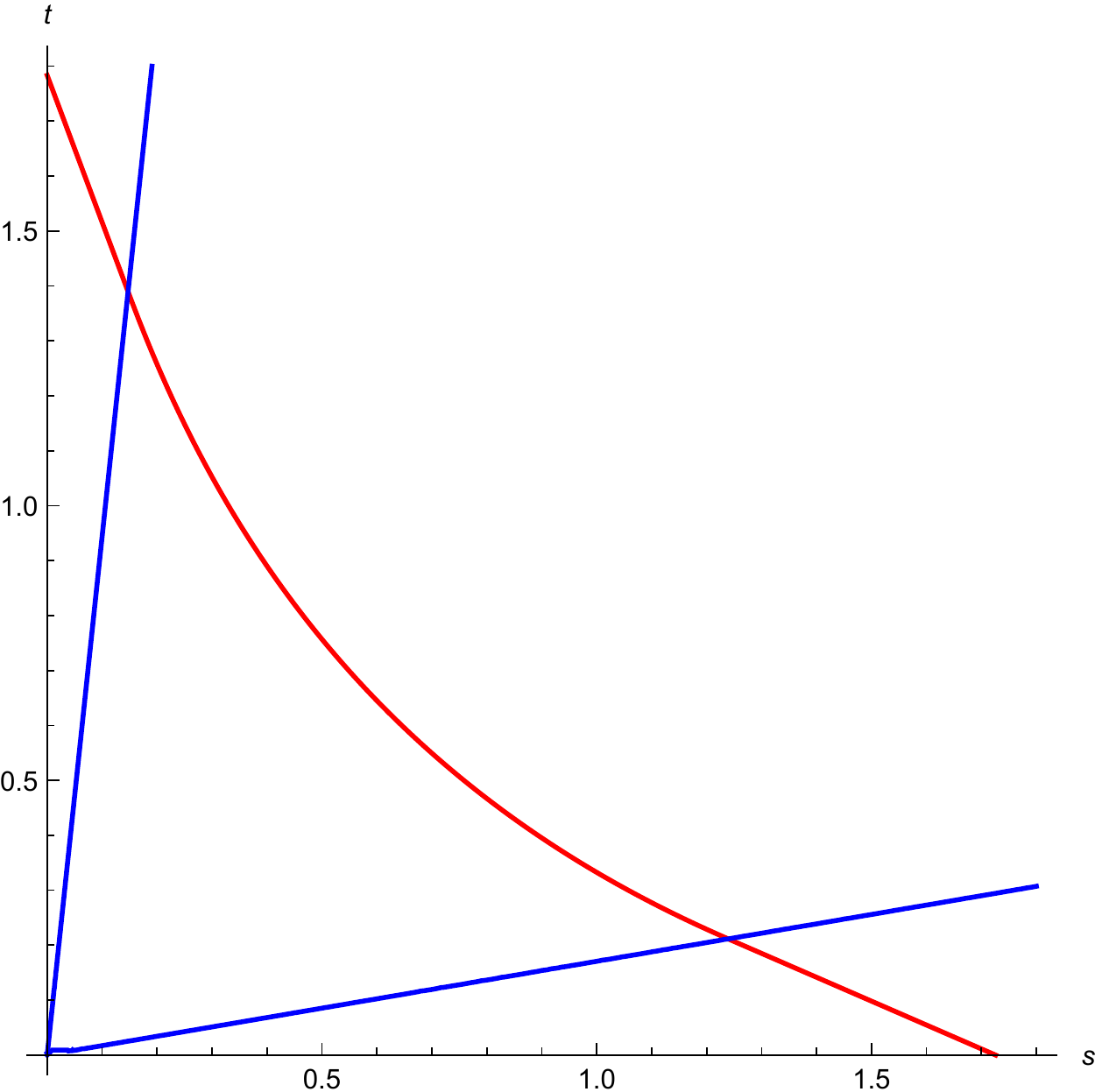}
\caption{The plot of $g = 1$ for (\ref{E7}) where $c_1 = 0$ and $c_2 = \infty$ (left). The plot of $g = 1$ and the rays $s/t = c_1 = (-8+12\log2)/3 \approx 0.105922$ and $s/t = c_2 = 4/(9-12\log2) \approx 5.863092$ for (\ref{E6}) when $\alpha(dx) = \one_{\{0 \le x \le 1\}}3x^2dx$ and $\beta(dx) = \one_{\{1 \le x \le 2\}} 4(x-1)^3dx$ (right).}
\label{F1}  
\end{figure}

A short discussion of the technical aspects of the paper is in order. To calculate the shape function, we rely on certain stationary processes with explicit product-form distributions. This approach dates back to \cite{Seppalainen2} and is illustrated in \cite{Seppalainen} to derive (\ref{E2}), which we briefly outline. Introduce a parameter $z \in (q, 1/q)$ and boundary weights $\{W(k, 0), W(0, k): k \in \bbZ_+\}$ such that $\{W(i, j): i, j \in \bbZ_+\}$ are independent, $W(0, 0) = 0$, and $W(k, 0)$ and $W(0, k)$ are geometrically distributed with parameters $q/z$ and $qz$, respectively. Choose the boundary values in recursion (\ref{e2}) as $G(i, 0) = \sum_{k=1}^{i} W(k, 0)$ and $G(0, j) = \sum_{k=1}^{j} W(0, k)$ for $i, j \in \bbN$. Then the resulting model turns out to be \emph{stationary} in the sense that the distributions of the processes $\{G(i, n)-G(i-1, n): i \in \bbN\}$ and $\{G(n, j)-G(n, j-1): j \in \bbN\}$ do not depend on $n$. Consequently, these distributions are product measures with geometric marginals. This allows computing the shape function $g_z$ of the stationary model and relating it to $g$ via
\begin{equation}\label{E8}g_z(1, 1) = \sup_{t \in [0, 1]} \max \{g_z(1-t, 0) + g(t, 1), g_z(0, 1-t) + g(1, t)\} \end{equation}
for $z \in (q, 1/q)$. Then (\ref{E2}) can be extracted from (\ref{E8}). The main observation of the present work is that, given $(\bfa, \bfb)$, if we choose the parameters of $W(i, 0)$ and $W(0, j)$ as $a_i/z$ and $b_j z$ for $z \in (\bar{\alpha}, 1/\bar{\beta})$, we still obtain a stationary model. Then, adapting some arguments from \cite{Seppalainen}, we also arrive at (\ref{E8}). To identify $g$, \cite{Seppalainen} uses the symmetry $g(s, t) = g(t, s)$, which is not true for the inhomogeneous model unless $\alpha = \beta$. For this step, we develop an argument that removes the need for symmetry and makes only a few general assumptions on $g$ and $g_z$. 

\textbf{Literature review.}
We mention briefly some related results and conjectures beginning with the case of i.i.d. $P$. For the exponential and geometric models,  
\begin{align}
\label{E3}
g(s, t) = m(s+t)+2 \sqrt{\sigma^2 st}, 
\end{align} 
where $m$ and $\sigma^2$ are the common mean and the variance of the weights. Furthermore, $g$ satisfies (\ref{E3}) up to an error of order $o(\sqrt{t})$ as $t \downarrow 0$ provided that the weights have sufficiently light tail \cite{Martin2}. There are also known cases in which $g$ has linear segments. For example, if the weights are bounded by $1$ and $P(W(i, j) = 1) = p > p_c$, where $p_c$ is the critical probability for the oriented site percolation, then $g(s, t) = s+t$ in a nontrivial cone \cite{AuffingerDamron}, \cite{DurrettLiggett}, \cite{Marchand}. Nevertheless, it is expected that $g$ is strictly concave and differentiable for a large class of $P$, for instance, when the weights have continuous distributions with enough moments. In our setting, $g$ also enjoys these properties under some moment conditions but can be far from (\ref{E3}) as (\ref{E7}) exemplifies. More recently, variational formulas in terms of stationary, integrable cocycles have been developed for $g$ under the mild assumption that the weights have finite $2+\epsilon$ moment for some $\epsilon > 0$ \cite{GeorSeppRass}. Minimizers of these formulas are identified as Busemann functions in \cite{GeorSeppRass2} relying on some fixed point results from queueing theory \cite{MairessePrabhakar}. 

There has also been interest to identify the shape function for non-i.i.d. $P$. For the exponential model with columnwise inhomogeneity (i.i.d. $\bfa$ and constant $\bfb$), \cite{KrugSepp} obtained a variational description of $g$, which (\ref{E6}) includes as a special case. Asymptotics of $g$ near the axes are determined for more general $P$ in \cite{Lin}. Their Theorem 2.4 can be deduced from (\ref{E6}).  Another direction of generalizing the exponential and geometric models is to choose the parameters at site $(i, j)$ as $\lambda = \Lambda(i/n, j/n)$ and $q = Q(i/n, j/n)$ for some deterministic functions $\Lambda$ and $Q$ that encode inhomogeneity. Then, under suitable conditions, $g$ can be characterized as the unique monotone viscosity solution of a certain Hamilton-Jacobi equation \cite{Calder}.

While we will not take advantage of it in the present paper, we mention that exact solvability of the exponential and geometric models goes beyond the explicit limits (\ref{E1}) and (\ref{E2}). The distributions of the last-passage times can be expressed as a Fredholm determinant with an explicit kernel. Using this, \cite{Johansson} established that $G(\lf ns \rf, \lf nt \rf)$ has fluctuations of order $n^{1/3}$ and converges weakly, after suitable rescaling, to the Tracy-Widom GUE distribution. These features are characteristic of the conjectural Kardar-Parisi-Zhang (KPZ) universality class, see survey \cite{Corwin}. More generally, Fredholm determinant representations of the distribution of the last-passage times under $\bfP_{\bfa, \bfb}$ have been derived by relating $\bfP_{\bfa, \bfb}$ to the Schur measures introduced in \cite{Okounkov} and, thereby, to determinantal point processes with explicit correlation kernels \cite{BorodinPeche}, \cite{Johansson2}, \cite{Johansson4}, \cite{Johansson3}. 

\textbf{Outline.} Our results are formally stated in Section \ref{S3}. 
We sketch the existence and the basic properties of $g$ in Section \ref{S2}. We discuss stationary versions of the exponential and geometric models in Section \ref{S4}. We prove (\ref{E6}) in Section \ref{S5}. 

\textbf{Notation and conventions.} $\bbN = \{1, 2, 3, \dotsc\}$ and $\bbZ_+ = \{0, 1, 2, \dotsc\}$. For $x \in \bbR$, define $\lf x \rf = \max \{i \in \bbZ: i \le x\}$ and $\lc x \rc = \min \{i \in \bbZ: i \ge x\}$. Also, $x \vee y = \max\{x, y\}$ and $x \wedge y = \min \{x, y\}$ for $x, y \in \bbR$. The imaginary unit is denoted by $\bfi$. Adjectives \emph{increasing} and \emph{decreasing} are used in the strict sense. For convenience, we set $1/0 = \infty$ and $1/\infty = 0$. 

\textbf{Acknowledgement. } The author thanks Timo Sepp\"{a}l\"{a}inen for the problem of computing explicit limit shapes and his valuable comments during the preparation of this paper. 

\section{Results}\label{S3}

Let $\E$ denote the expectation under $\mu$ (the distribution of $(\bfa, \bfb)$). Recall that $a = a_1$ and $b = b_1$. It is convenient to break (\ref{E6}) into the next two theorems. 

\begin{thm}
\label{T1}
Suppose that $\ubar{\alpha} + \ubar{\beta} > 0$ in the exponential model. Then
\begin{equation}\label{eq3}
g(s, t) = \inf \limits_{z \in (-\ubar{\alpha}, \ubar{\beta})}\  \left\{s\E\left[\frac{1}{a+z}\right] + t\E\left[\frac{1}{b-z}\right]\right\} \quad \text{ for } s, t > 0. \end{equation}
\end{thm}
Hence, $g$ depends on $(\bfa, \bfb)$ only through the marginal distributions $\alpha$ and $\beta$. Let us write $g^{\alpha, \beta}$ to indicate this. Replacing $z$ with $-z$ in (\ref{eq3}) reveals that $g^{\alpha, \beta}(s, t) = g^{\beta, \alpha}(t, s)$ for $s, t > 0$, which is expected due to the symmetric roles of $\bfa$ and $\bfb$ in the model. In particular, if $\alpha$ and $\beta$ are the same then $g(s, t) = g(t, s)$ for $s, t > 0$. Also, (by dominated convergence) the infimum can be taken over $[-\ubar{\alpha}, \ubar{\beta}]$ in (\ref{eq3}). When $\ubar{\alpha} = \ubar{\beta} = 0$, this interval degenerates to $\{0\}$ and we expect that $g(s, t) = s\E[1/a]+t\E[1/b]$ for $s, t > 0$. Indeed, this is true.  
\begin{thm}
\label{T2}
Suppose that $\ubar{\alpha} = \ubar{\beta} = 0$ in the exponential model. Then 
\begin{equation*}
g(s, t) = s\E\left[\frac{1}{a}\right] + t\E\left[\frac{1}{b}\right] \text{ for } s, t > 0. 
\end{equation*}
\end{thm}

We turn to the concavity and differentiability properties of $g$. In the case $\ubar{\alpha}+\ubar{\beta} > 0$, define the critical values
$c_1 = \dfrac{\E[(b+\ubar{\alpha})^{-2}]}{\E[(a-\ubar{\alpha})^{-2}]}$ and $c_2 = \dfrac{\E[(b-\ubar{\beta})^{-2}]}{\E[(a+\ubar{\beta})^{-2}]}$.  
Note that $0 \le c_1 < c_2 \le \infty$. Also, $c_1 = 0$ if and only if $\E[(a-\ubar{\alpha})^{-2}] = \infty$, and $c_2 = \infty$ if and only if $\E[(b-\ubar{\beta})^{-2}] = \infty$.
\begin{cor}
\label{c3}
Suppose that $\ubar{\alpha} + \ubar{\beta} > 0$ in the exponential model. Then 
\begin{enumerate}[(a)] 
\item $g(s, t) = s\E[(a-\ubar{\alpha})^{-1}] + t\E[(b+\ubar{\alpha})^{-1}]$ for $s/t \le c_1$. 
\item $g(s, t) = s\E[(a+\ubar{\beta})^{-1}] + t\E[(b-\ubar{\beta})^{-1}]$ for $s/t \ge c_2$. 
\item $g(cs_1+(1-c)s_2, ct_1+(1-c)t_2) > cg(s_1, t_1) + (1-c)g(s_2, t_2)$ for $c \in (0, 1)$ and $s_1, s_2, t_1, t_2 > 0$ such that $c_1 < s_1/t_1, s_2/t_2 < c_2$ and $(s_1, t_1) \neq k (s_2, t_2)$ for any $k \in \bbR$. 
\item $g$ is continuously differentiable.
\end{enumerate}
\end{cor}
By Schwarz inequality, if $c_1 > 0$ then $\E[(a-\ubar{\alpha})^{-1}] < \infty$ and if $c_2 < \infty$ then $\E[(b-\ubar{\beta})^{-1}] < \infty$. Hence, $g$ is finite and linear in $(s, t)$ in the regions $s/t \le c_1$ and $s/t \ge c_2$.  
\begin{proof}[Proof of Corollary \ref{c3}]
Let $A(z) = \E[(a+z)^{-1}]$ for $z > -\ubar{\alpha}$ and $B(z) = \E[(b-z)^{-1}]$ for $z < \ubar{\beta}$. Using dominated convergence, $A$ and $B$ can be differentiated under the expectation. Thus, 
$A'(z) = -\E[(a+z)^{-2}],\ B'(z) = \E[(b-z)^{-2}],\ A''(z) = 2\E[(a+z)^{-3}],\ B''(z) = 2\E[(b-z)^{-3}]$, etc. Also, define $A, B$ and their derivatives at the endpoints by substituting $-\ubar{\alpha}$ and $\ubar{\beta}$ for $z$ in the preceding formulas. Then, by monotone convergence, the values at the endpoints match the appropriate one-sided limits, that is, $A(-\ubar{\alpha}) = \lim_{z \downarrow -\ubar{\alpha}} A(z) = \E[(a-\ubar{\alpha})^{-1}]$, $B(\ubar{\beta}) = \lim_{z \uparrow \ubar{\beta}} B(z) = \E[(b-\ubar{\beta})^{-1}]$, and similarly for the derivatives. 

Since $A'$ and $B'$ are increasing and continuous on $(-\ubar{\alpha}, \ubar{\beta})$, the derivative $z \mapsto sA'(z) + tB'(z)$ is positive if $s/t \le c_1$, is negative if $s/t \ge c_2$ and has a unique zero if $c_1 < s/t < c_2$. Hence, (a) and (b) follow, and if $c_1 < s/t < c_2$ then $g(s, t) = sA(z) + tB(z)$, where 
$z \in (-\ubar{\alpha}, \ubar{\beta})$ is the unique solution of the equation
\begin{equation}\label{eq8}-\frac{B'(z)}{A'(z)} = \frac{s}{t}.\end{equation}
Since $ -B'/A'$ is increasing and continuous, it has an increasing inverse $\zeta$ defined on $(c_1, c_2)$. Let $s_1, t_1, s_2, t_2$ be as in (c). 
Then $\zeta(s_1/t_1) \neq \zeta(s_2/t_2)$, which implies the strict inequality  
\begin{equation}
\label{eq100}
\begin{aligned}
(s_1+s_2)A(z) + (t_1+t_2)B(z) &> s_1A(\zeta(s_1/t_1)) + t_1B(\zeta(s_1/t_1))\\ 
&+s_2A(\zeta(s_2/t_2)) + t_2B(\zeta(s_2/t_2)) \\
&= g(s_1, t_1) + g(s_2, t_2) \\
\end{aligned}
\end{equation}   
for any $z \in (-\ubar{\alpha}, \ubar{\beta})$. Note that $c_1 < (s_1+s_2)/(t_1+t_2) < c_2$. Setting
$z = \zeta((s_1+s_2)/(t_1+t_2))$ in (\ref{eq100}) yields  $g(s_1+s_2, t_1+t_2) > g(s_1, t_1) + g(s_2, t_2)$, and (c) comes from this and homogeneity. Since $-B'/A'$ is continuously differentiable with positive derivative
(as $A'', B', B > 0$ and $A' < 0$ on $(-\ubar{\alpha}, \ubar{\beta})$), by the inverse function theorem, $\zeta$ is continuously differentiable as well. Using (\ref{eq8}), we compute the gradient of $g$ for $c_1 < s/t < c_2$ as 
$
\nabla g(s, t) =  (A(\zeta(s/t)), B(\zeta(s/t))), 
$
which tends to $(A(-\ubar{\alpha}), B(-\ubar{\alpha}))$ as $s/t \rightarrow c_1$ and to $(A(\ubar{\beta}), B(\ubar{\beta}))$ as $s/t \rightarrow c_2$. Hence, (d). 
\end{proof}

When $\alpha$ and $\beta$ are uniform distributions, we can compute the infimum in (\ref{eq3}) explicitly.  
\begin{cor}
\label{C1}
Let $\lambda, l, m > 0$. Suppose that $\alpha$ and $\beta$ are uniform distributions on $[\lambda/2, \lambda/2+l]$ and $[\lambda/2, \lambda/2+m]$, respectively. 
Then, for $s, t > 0$,  
\begin{align*}
g(s, t) &= \frac{s}{l}\log\left(1+\frac{l}{\lambda} + \frac{l}{\lambda} \cdot \frac{lt-ms + \sqrt{(lt-ms)^2+4st(\lambda+l)(\lambda+m)}}{2s(\lambda+m)}\right)\\
&+ \frac{t}{m}\log\left(1+\frac{m}{\lambda} + \frac{m}{\lambda} \cdot \frac{ms-lt + \sqrt{(lt-ms)^2+4st(\lambda+l)(\lambda+m)}}{2t(\lambda+l)}\right). 
\end{align*}
\end{cor}
\begin{proof}
Since $\alpha$ and $\beta$ are uniform distributions,
\begin{align*}
A(z) = \E\left[\frac{1}{a+z}\right] = \frac{1}{l}\log\bigg(1+\frac{l}{z+\lambda/2}\bigg) \quad B(z) = \E\left[\frac{1}{b-z}\right] = \frac{1}{m}\log\bigg(1+\frac{m}{-z+\lambda/2}\bigg)
\end{align*} 
for $z \in (-\lambda/2, \lambda/2)$. We compute the derivatives as 
\begin{align*}
A'(z) = -\frac{1}{(z+\lambda/2)(z+\lambda/2+l)} \quad B'(z) = \frac{1}{(-z+\lambda/2)(-z+\lambda/2+m)}.
\end{align*}
Because $A'(-\lambda/2) = -\infty$ and $B'(\lambda/2) = \infty$, we have $c_1 = 0$ and $c_2 = \infty$. Also, (\ref{eq8}) leads to 
\begin{align*}
(s-t)z^2-(s(\lambda+m)+t(\lambda+l))z +s\lambda(\lambda+2m)/4-t\lambda(\lambda+2l)/4 = 0. 
\end{align*}
It follows from the discriminant formula that the solution in the interval $(-\lambda/2, \lambda/2)$ is 
\begin{align*}
z = \frac{\lambda}{2} \frac{s(\lambda+2m)-t(\lambda+2l)}{s(\lambda+m)+t(\lambda+l)+ \sqrt{(sm+tl)^2+4st\lambda(\lambda+m+l})}. 
\end{align*}
Inserting this into $g(s, t) = sA(z)+tB(z)$ and some elementary algebra yield the result.  
\end{proof}
The preceding argument can be repeated when $l = 0$ or $m = 0$. In these cases, $\alpha$ and $\beta$ are understood as point masses at $\lambda/2$. For instance, when $l = 0$ and $m > 0$, we obtain 
\begin{align*}
g(s, t) &= \frac{2s\lambda+ms+\sqrt{(ms)^2+4st\lambda(\lambda+m)}}{2\lambda(\lambda+m)}+\frac{t}{m}\log\left(1+\frac{m}{\lambda} + \frac{m}{\lambda} \cdot \frac{ms + \sqrt{(ms)^2+4st\lambda(\lambda+m)}}{2t\lambda}\right)
\end{align*}
When $l = 0$ and $m = 0$, we recover (\ref{E1}). 

We can also determine $g$ along the diagonal when $\alpha$ and $\beta$ are the same. 
\begin{cor}
Suppose that $\alpha = \beta$. Then $g(s, s) = 2s\E\left[\dfrac{1}{a}\right]$ for $s > 0$. 
\end{cor}
\begin{proof}
We have 
$
(a+z)^{-1}+(a-z)^{-1} \ge 2a^{-1}
$
for $|z| \le \ubar{\alpha}$ with equality if only if $z = 0$. Therefore, 
\begin{align*}g(s, s) = s \inf_{z \in (-\ubar{\alpha}, \ubar{\alpha})} \E\left[\frac{1}{a+z}+\frac{1}{a-z}\right] = 2s \E\left[\frac{1}{a}\right]. \qquad \qedhere
\end{align*}
\end{proof}

We only report the analogous results for the geometric model. 
\begin{thm}
Suppose that $\bar{\alpha}\bar{\beta} < 1$ in the geometric model. Then
\begin{equation*}
g(s, t) = \inf \limits_{z \in (\bar{\alpha}, 1/\bar{\beta})}\  \left\{s\E\left[\frac{a/z}{1-a/z}\right] + t\E\left[\frac{bz}{1-bz}\right]\right\} \quad \text{ for } s, t > 0. \end{equation*}
\end{thm}

\begin{thm}
Suppose that $\bar{\alpha} = \bar{\beta} = 1$ in the geometric model. Then 
\begin{equation*}
g(s, t) = s\E\left[\frac{a}{1-a}\right] + t\E\left[\frac{b}{1-b}\right] \quad \text{ for } s, t > 0. 
\end{equation*}
\end{thm}

\begin{cor}
Let $q \in (0, 1)$ and $0 < l, m < \sqrt{q}$. Choose $\alpha$ and $\beta$ as the distributions with densities proportional to $x \mapsto 1/x$ on the intervals $[\sqrt{q}-l, \sqrt{q}]$ and $[\sqrt{q}-m, \sqrt{q}]$, respectively. 
Then 
\begin{align*}
g(s, t) &= \frac{s}{L}\log\left(1 + \frac{l\sqrt{q}}{1-q}+\frac{l}{1-q}\frac{ly-mx+\sqrt{(ly-mx)^2+4xy(1+m\sqrt{q}-q)(1+l\sqrt{q}-q)}}{2x(1+m\sqrt{q}-q)}\right)\\
&+ \frac{t}{M}\log\left(1 + \frac{m\sqrt{q}}{1-q}+\frac{m}{1-q}\frac{mx-ly+\sqrt{(ly-mx)^2+4xy(1+m\sqrt{q}-q)(1+l\sqrt{q}-q)}}{2y(1+l\sqrt{q}-q)}\right)
\end{align*}
for $s, t > 0$, where $x = slM$, $y = tmL$, $L = \log\left(\dfrac{\sqrt{q}}{\sqrt{q}-l}\right)$ and $M = \log\left(\dfrac{\sqrt{q}}{\sqrt{q}-m}\right)$.
\end{cor}

\begin{cor}
Suppose that $\alpha = \beta$. Then $g(s, s) = 2s \E\left[\dfrac{a}{1-a}\right]$ for $s > 0$. 
\end{cor}

\section{The existence of the shape function} \label{S2}
\begin{lem}
\label{L1}
There exists a deterministic function $g:(0, \infty)^2 \rightarrow [0, \infty]$ such that  
\begin{equation*}
\lim \limits_{n \rightarrow \infty} \frac{G(\lf ns \rf, \lf nt \rf)}{n} = g(s, t) \quad \text{ for } s, t > 0 \quad \bbP\text{-a.s.}
\end{equation*}
Furthermore, $g$ is nondecreasing, homogeneous and concave. 
\end{lem}
Here, nondecreasing means that $g(s', t') \le g(s, t)$ for $0 < s' \le s$ and $0 < t' \le t$, and homogeneity means that $g(cs, ct) = cg(s, t)$ for $s, t, c > 0$. In the exponential model, $g$ is finite if $\ubar{\alpha}+\ubar{\beta} > 0$. This is by the standard properties of the stochastic order \cite[Theorem 1.A3]{Shaked}. Briefly, the i.i.d. measure $P$ on $\bbR_+^{\bbN^2}$ under which each $W(i, j)$ is exponentially distributed with rate $\ubar{\alpha}+\ubar{\beta}$ stochastically dominates $\bfP_{\bfa, \bfb}$ and $G$ is a nondecreasing function of the weights. Thus, $g(s, t)$ does not exceed the right-hand side of (\ref{E1}) with $\lambda = \ubar{\alpha}+\ubar{\beta}$. Similarly, $g$ is finite in the geometric model if $\bar{\alpha}\bar{\beta} < 1$. Extend $g$ to $\bbR_+^2$ by setting $g(0, 0) = 0$, $g(s, 0) = \lim_{t \downarrow 0} g(s, t)$ and $g(0, t) = \lim_{s \downarrow 0} g(s, t)$ for $s, t > 0$. 

Lemma \ref{L1} can be proved using the ergodicity properties of $\bbP$ and superadditivity of the last-passage times. As this is quite standard, we will leave out many details. For $k, l \in \bbZ_+$, let $\theta_{k, l}: \bbR^{\bbN^2} \rightarrow \bbR^{\bbN^2}$ be given by $\theta_{k, l}(\omega)(i, j) = \omega(i+k, j+l)$ for $i, j \in \bbN$ and $\omega \in \bbR^{\bbN^2}$. Note that $\bbP$ is stationary with respect to $\theta_{k, l}$ because $\bbP(\theta_{k, l}^{-1}(B)) = \E\bfP_{\bfa, \bfb}(\theta_{k, l}^{-1}(B)) = \E\bfP_{\tau_k (\bfa), \tau_l (\bfb)}(B) = \bbP(B)$ for any Borel set $B \subset \bbR_+^{\bbN^2}$.
\begin{lem}
\label{L2}
$\bbP$ is ergodic with respect to $\theta_{k, l}$ for any $k, l \in \bbN$. 
\end{lem} 
\begin{proof}
Suppose $\theta_{k, l}^{-1}(B) = B$ for some Borel set $B \subset \bbR_+^{\bbN^2}$. For $n \ge 1$, let $\sT_n$ denote the $\sigma$-algebra generated by $A_n$, the collection of $W(i, j)$ with $i > k(n-1)$ and $j > l(n-1)$. Then $B$ is in $\sT = \bigcap_{n \in \bbN} \sT_n$. Also, $\sT$ is the tail $\sigma$-algebra of the $\sigma$-algebras generated by $A_{n}\smallsetminus A_{n+1}$. Because $\bfP_{\bfa, \bfb}$ is a product measure, by Kolmogorov's $0$--$1$ law,
$\bfP_{\bfa, \bfb}(B) \in \{0, 1\}$. Therefore, 
$\bbP(B) = \mu(\bfP_{\bfa, \bfb}(B) = 1)$. 
On the other hand, 
\[\begin{aligned}
(\tau_k \times \tau_l)^{-1}\{\bfP_{\bfa, \bfb}(B) = 1\} &= \{\bfP_{\tau_k \bfa, \tau_l \bfb}(B) = 1\} = \{\bfP_{\bfa, \bfb}(\theta_{k, l}^{-1}(B)) = 1\} = \{\bfP_{\bfa, \bfb}(B) = 1\}. 
\end{aligned}\]
Since $\mu$ is ergodic under $\tau_k \times \tau_l$, we conclude that $\bbP(B) \in \{0, 1\}$.  
\end{proof}

\begin{proof}[Proof of Lemma \ref{L1}]
Fix $s, t \in \bbN$ and define, for integers $0 \le m < n$,  
\begin{align*}Z(m, n) = -G((n-m)s, (n-m)t) \circ \theta_{ms, mt} = \max \limits_{\pi \in \Pi_{ms+1, mt+1, ns, nt}} \sum \limits_{(i, j) \in \pi} W(i, j).\end{align*}
Using the definition and Lemma \ref{L2}, we observe that $\{Z(m, n): 0 \le m < n\}$ is a subadditive process that satisfies the hypotheses of Liggett's subadditive ergodic theorem \cite{Liggett}.  
Hence, $Z(0, n)/n = G(ns, nt)/n$ converges $\bbP$-a.s. to a deterministic limit, $g(s, t)$. The existence of the 
limit for all $s, t > 0$ $\bbP$-a.s. and the claimed properties of $g$ follow as in the case of i.i.d. weights  \cite[Theorem~2.1]{Seppalainen}. 
\end{proof}

\section{Stationary distributions of the last-passage increments}  
\label{S4}


Let us extend the sample space to $\bbR_+^{\bbZ_+^2}$. Now $W(i, j)$ denotes the projection onto coordinate $(i, j)$ for $i, j \in \bbZ_+$. Define the last-passage time $\widehat{G}(i, j)$ through recursion (\ref{e2}) but with the boundary values 
$\widehat{G}(i, 0) = \sum_{k=1}^i W(k, 0)$ and $\widehat{G}(0, j) = \sum_{k=1}^j W(0, k)$ for $i, j \in \bbN$. 
We then have
\begin{equation}
\label{eq42}
\widehat{G}(m, n) = \max \limits_{\pi \in \Pi_{0, 0, m, n}} \sum \limits_{(i, j) \in \pi} W(i, j) \quad \text{ for } m, n \in \bbZ_+. 
\end{equation}

In the exponential model, 
for each value of $(\bfa, \bfb)$ such that $a_n \ge \ubar{\alpha}$ and $b_n \ge \ubar{\beta}$ for $n \in \bbN$ (which holds $\mu$-a.s.) and parameter $z \in (-\ubar{\alpha}, \ubar{\beta})$, define $\bfP_{\bfa, \bfb}^z$ as the product measure on $\bbR_+^{\bbZ_+^2}$ by 
\begin{equation}
\label{eq78}
\begin{aligned}
&\bfP_{\bfa, \bfb}^z(W(i, j) \ge x) = \exp(-(a_i+b_j)x) \quad &&\bfP_{\bfa, \bfb}^z(W(0, 0) = 0) = 1 \\
&\bfP_{\bfa, \bfb}^z(W(i, 0) \ge x) = \exp(-(a_i+z)x) \quad &&\bfP_{\bfa, \bfb}^z(W(0, j) \ge x) = \exp(-(b_j-z)x)
\end{aligned}
\end{equation}
for $x \ge 0$ and $i, j \in \bbN$. When $\ubar{\alpha} = \ubar{\beta} = 0$, we make definition (\ref{eq78}) for $z = 0$. Note that the projection of $\bfP_{\bfa, \bfb}^z$ onto coordinates $\bbN^2$ is $\bfP_{\bfa, \bfb}$. 
For the geometric model, the construction is similar. 
For $z \in (\bar{\alpha}, 1/\bar{\beta})$ and each value of $(\bfa, \bfb)$ such that $a_n \le \bar{\alpha}$ and $b_n \le \bar{\beta}$ for $n \in \bbN$, the measure $\bfP_{\bfa, \bfb}^z$ is given by  
\begin{equation}
\label{eq79}
\begin{aligned}
&\bfP_{\bfa, \bfb}^z(W(i, j) \ge k) = a_i^k b_j^k \quad &&\bfP_{\bfa, \bfb}^z(W(0, 0) = 0) = 1 \\
&\bfP_{\bfa, \bfb}^z(W(i, 0) \ge k) = a_i^k/z^k \quad &&\bfP_{\bfa, \bfb}^z(W(0, j) \ge k) = b_j^kz^k
\end{aligned}
\end{equation}
for $k \in \bbZ_+$ and $i, j \in \bbN$. When $\bar{\alpha} = \bar{\beta} = 1$, definition (\ref{eq79}) makes sense for $z = 1$. 
 
Introduce the increment variables as $I(m, n) = \widehat{G}(m, n)-\widehat{G}(m-1, n)$ for $m \ge 1$ and $n \ge 0$, and $J(m, n) = \widehat{G}(m, n)-\widehat{G}(m, n-1)$ for $m \ge 0$ and $n \ge 1$. We capture the stationarity of the increments in the following proposition. 
\begin{prop}
\label{p5}
Let $k, l \in \bbZ_+$. Under $\bfP_{\bfa, \bfb}^z$,   
\begin{enumerate}[(a)]
\item $I(i, l)$ has the same distribution as $W(i, 0)$ for $i \in \bbN$. 
\item $J(k, j)$ has the same distribution as $W(0, j)$ for $j \in \bbN$. 
\item The random variables $\{I(i, l): i > k\} \cup \{J(k, j): j > l\}$ are (jointly) independent. 
\end{enumerate}
\end{prop}

(\ref{e2}) leads to the recursion \cite[(2.21)]{Seppalainen}
\begin{equation}
\label{eq81}
\begin{aligned}
I(m, n) &= I(m, n-1) - I(m, n-1) \wedge J(m-1, n) + W(m, n) \\
J(m, n) &= J(m-1, n) - I(n, n-1) \wedge J(m-1, n) + W(m, n) 
\end{aligned}
\end{equation}
for $m, n \in \bbN$. Proposition \ref{p5} can be proved via induction using (\ref{eq81}) and Lemma \ref{l6} below. We will omit the induction argument as it is the same as in \cite[Theorem~2.4]{Seppalainen}.  
 \begin{lem}
\label{l6}
 Let $F: \bbR^3 \rightarrow \bbR^3$ denote the map 
$(x, y, z) \mapsto (x-x \wedge y + z, y - x \wedge y + z, x \wedge y)$. 
Let $P$ be a product measure on $\bbR^3$ with marginals $P_1, P_2, P_3$. Suppose that one of the following holds. 
\begin{enumerate}[(\romannumeral1)]
\item $P_1, P_2$ and $P_3$ are exponential distributions with rates $a, b$ and $a+b$, for some $a, b \in (0, \infty)$. 
\item $P_1, P_2$ and $P_3$ are geometric distributions with parameters $a, b$ and $ab$, for some $a, b \in (0, 1)$. 
\end{enumerate}
Then $P( F^{-1}(B)) = P(B)$ for any Borel set $B \subset \bbR^3$. 
 \end{lem}
In earlier work \cite[Lemma~2.3]{Seppalainen} and \cite[Lemma~4.1]{Balasz}, Lemma \ref{l6} was proved by comparing the Laplace transforms of the measures $P$ and $P(F^{-1}(\cdot))$. 
We include another proof below. 
 \begin{proof}[Proof of Lemma \ref{l6}]
We prove (\romannumeral1) only as the proof of (\romannumeral2) is the discrete version of the same argument and is simpler. Observe that $F$ is a bijection on $\bbR^3$ with $F^{-1} = F$. 
It suffices to verify the claim for any open set $B$ in $\bbR^3$.  By continuity, $F^{-1}(B)$ is also open. Furthermore, $F$ is differentiable on the open set $\{(x, y, z): x > y \text{ or } x < y\}$ and its Jacobian equals $1$ in absolute value. Hence, by the change of variables \cite[Theorem 7.26]{Rudin}, 
\[
\begin{aligned}
P(F^{-1}(B)) &= ab(a+b)\int \limits_{F^{-1}(B)} e^{-ax-by-(a+b)z} dx dy dz \\
&= ab(a+b)\int \limits_{F^{-1}(B)} e^{-a(x-x \wedge y + z)-b(y-x \wedge y + z) -(a+b)(x \wedge y)} dx dy dz \\ 
&= ab(a+b)\int \limits_{B} e^{-au-bv-(a+b)w} du dv dw = P(B). \qedhere
\end{aligned}
\]

 \end{proof}

In the exponential and geometric models, respectively, define 
\begin{align*}g_z(s, t) &= s\E\left[\frac{1}{a+z}\right]+t\E\left[\frac{1}{b-z}\right] \quad \text{ for } s, t \ge 0 \text{ and } z \in [-\ubar{\alpha}, \ubar{\beta}]\\
g_z(s, t) &= s\E\left[\frac{a/z}{1-a/z}\right]+t\E\left[\frac{bz}{1-bz}\right] \quad \text{ for } s, t \ge 0 \text{ and } z \in [\bar{\alpha}, 1/\bar{\beta}]. 
\end{align*} 

\begin{lem}
\label{L4}
In the exponential model, let $z \in (-\ubar{\alpha}, \ubar{\beta})$ if $\ubar{\alpha}+\ubar{\beta} > 0$, and let $z = 0$ and assume that $\E[1/a+1/b] < \infty$ if $\ubar{\alpha} = \ubar{\beta} = 0$. In the geometric model, let $z \in (\bar{\alpha}, 1/\bar{\beta})$ if $\bar{\alpha}\bar{\beta} < 1$, and let $z = 1$ and assume that $\E[a/(1-a)+b/(1-b)] < \infty$ if $\bar{\alpha} = \bar{\beta} = 1$. Then 
\begin{align}\lim \limits_{n \rightarrow \infty} \frac{\widehat{G}(\lf ns \rf, \lf nt \rf)}{n} = g_z(s, t) \quad \text{ for } s, t \ge 0 \text{ in } \bfP_{\bfa, \bfb}^z\text{-probability for } \mu\text{-a.e. } (\bfa, \bfb). \label{E11}
\end{align} 
\end{lem}
In fact, the convergence in (\ref{E11}) is $\bfP_{\bfa, \bfb}^z$-a.s. for $\mu$-a.e $(\bfa, \bfb)$ provided that $\ubar{\alpha}+\ubar{\beta} > 0$ in the exponential model and $\bar{\alpha}\bar{\beta} < 1$ in the geometric model \cite[Theorem~4.3]{Emrah}. 
By (\ref{eq51}), (\ref{eq42}) and nonnegativity of weights, $G(m, n) \le \widehat{G}(m, n)$ for $m, n \in \bbN$. Then Lemma \ref{L4} implies that $g(s, t) \le g_z(s, t)$ for any $s, t \ge 0$. The main result of this paper is that $g(s, t) = \inf_z g_z(s, t)$. 

\begin{proof}[Proof of Lemma \ref{L4}]
We will consider the exponential model only, the geometric model is treated similarly.   
Note that $\widehat{G}(\lf ns \rf, \lf nt \rf) = \sum_{i=1}^{\lf ns \rf} I(i, 0) + \sum_{j = 1}^{\lf nt \rf} J(\lf ns \rf, j)$ for $s, t \ge 0$ and $n \in \bbN$. By Proposition \ref{p5}, $\{J(\lf ns \rf, j): j \in \bbN\}$ has the same distribution as $\{J(0, j): j \in \bbN\}$ under $\bfP_{\bfa, \bfb}^z$. Hence, it suffices to show that 
\begin{align*}
\lim_{n \rightarrow \infty} \frac{1}{n} \sum \limits_{i=1}^{n} I(i, 0) = \E\left[\frac{1}{a+z}\right] \text{ and }  \lim_{n \rightarrow \infty} \frac{1}{n} \sum \limits_{j = 1}^{n} J(0, j) = \E\left[\frac{1}{b-z}\right] 
\end{align*}
in $\bfP_{\bfa, \bfb}^z$ for $\mu$-a.s. We will only derive the first limit above, for which we will show that, for $z > -\ubar{\alpha}$ and for $z = -\ubar{\alpha}$ when $\E[(a-\ubar{\alpha})^{-1}] < \infty$,  
\begin{align}
\lim_{n \rightarrow \infty} \frac{1}{n} \sum \limits_{i=1}^{n} I(i, 0) = \E\left[\frac{1}{a+z}\right] \quad \text{ in } \bfQ_{\bfa}^z \ \mu\text{-a.s.}, \label{E16}
\end{align}
where $\bfQ_{\bfa}^z$ is the product measure on the coordinates $\bbN \times \{0\}$ given by $\bfQ_{\bfa}^z(W(i, 0) \ge x) = e^{-(a_i+z)x}$ for $i \in \bbN$ and $x \ge 0$.
It suffices to prove the convergence in distribution under $\bfQ_{\bfa}^z$ $\mu$-a.s. because the limit is deterministic.  

The characteristic function of
$n^{-1}\sum_{i=1}^n I(i, 0)$ under $\bfQ_{\bfa}^z$ is given by 
\begin{align*}\prod \limits_{i=1}^n \left(1-\frac{\bfi x}{n(a_i+z)}\right)^{-1} = \exp\left(-\sum \limits_{i=1}^n \log\left(1-\frac{\bfi x}{n(a_i+z)}\right)\right) \quad \text{ for } x \in \bbR, \end{align*}
where the complex logarithm denotes the principal branch. Hence, (\ref{E16}) follows if we prove   
\begin{align*}
\lim_{n \rightarrow \infty} -\sum \limits_{i=1}^n \log\left(1-\frac{\bfi x}{n(a_i+z)}\right) = \bfi x \E\left[\frac{1}{a+z}\right]\quad \text{ for } x \in \bbR \quad \mu\text{-a.s.,}
\end{align*}
Using the bound $|\log(1+\bfi x)| \le |x|$ for $x \in \bbR$ 
and the ergodicity of $\bfa$, we obtain 
\begin{align*}\limsup_{n \rightarrow \infty} \left|\sum \limits_{i=1}^n \log\left(1-\frac{\bfi x}{n(a_i+z)}\right)\right| \le \lim_{n \rightarrow \infty} \frac{|x|}{n} \sum \limits_{i=1}^n \frac{1}{a_i+z} = |x| \E\left[\frac{1}{a+z}\right] \quad \text{ for } x \in \bbR \quad \mu\text{-a.s.}\end{align*} 
Therefore, it suffices to prove the following for $x \in \bbR$ $\mu$-a.s. 
\begin{align}\lim_{n \rightarrow \infty}-\sum \limits_{i=1}^n \arg\left(1-\frac{\bfi x}{n(a_i+z)}\right) - x\E\left[\frac{1}{a+z}\right] = \lim_{n \rightarrow \infty} \sum \limits_{i=1}^n \arctan\left(\frac{x}{n(a_i+z)}\right) - \frac{x}{n (a_i+z)} = 0.\label{E10}\end{align}
Since $\arctan x = \int_0^x (1+u^2)^{-1}du$, we can rewrite the second sum above as 
\begin{align*}
\sum \limits_{i=1}^n \int \limits_0^{xn^{-1}(a_i+z)^{-1}} \frac{du}{1+u^2}-\frac{x}{n (a_i+z)} &= -\sum \limits_{i=1}^n \int \limits_0^{xn^{-1}(a_i+z)^{-1}} \frac{u^2\ du}{1+u^2} 
= -\frac{x}{n} \sum \limits_{i=1}^n \int \limits_0^{(a_i+z)^{-1}} \frac{x^2v^2\ dv}{n^2+x^2v^2}, 
\end{align*}
where we changed the variables via $u = vx/n$. Pick $M > 0$. The limsup as $n \rightarrow \infty$ of the absolute value of the last sum is bounded $\mu$-a.s. by $|x|$ times  
\begin{align*}
\lim_{n \rightarrow \infty} \frac{1}{n} \sum \limits_{i=1}^n \int \limits_{0}^{(a_i+z)^{-1}} \frac{x^2v^2\ dv}{M^2+x^2v^2} =  \E\left[\int \limits_{0}^{(a+z)^{-1}} \frac{x^2 v^2 \ dv}{M^2+x^2v^2} \right],  
\end{align*}
where the a.s. convergence is due to the ergodicity of $\bfa$ and the integrability of 
\begin{align*}\int_0^{(a+z)^{-1}} \dfrac{x^2v^2 \ dv}{M^2+x^2v^2} \le \frac{1}{a+z}.\end{align*} 
The last integral is monotone in $x^2$ and vanishes as $M \rightarrow \infty$. Hence, (\ref{E10}) holds for $x \in \bbR$ $\mu$-a.s.
\end{proof}

The next proposition relates $g_z$ to $g$ through a variational formula. 
\begin{prop}
\begin{equation}\label{E8.1}g_z(1, 1) = \sup_{t \in [0, 1]} \max \{g_z(1-t, 0) + g(t, 1), g_z(0, 1-t) + g(1, t)\} \end{equation} For $z \in (-\ubar{\alpha}, \ubar{\beta})$ in the exponential model and for $z \in (\bar{\alpha}, 1/\bar{\beta})$ in the geometric model. 
\end{prop}
\begin{proof}
Fix $z \in (-\ubar{\alpha}, \ubar{\beta})$ in the exponential model. Since $g \le g_z$ and $g_z$ is linear, (\ref{E8.1}) with $\ge$ instead of $=$ is immediate. For the opposite inequality, we adapt the argument in \cite[Proposition~2.7]{Seppalainen}. It follows from (\ref{eq51}) and (\ref{eq42}) that  
\begin{equation}
\label{eq23}
\widehat{G}(n, n) = \max \limits_{k \in [n]} \max \{\widehat{G}(k, 0) + G(n-k+1, n) \circ \theta_{k-1, 0}, \widehat{G}(0, k) + G(n, n-k+1) \circ \theta_{0, k-1}\}. 
\end{equation}
Let $L \in \bbN$ and consider $n > L$ large enough so that $\lc (i+1)n/L\rc > \lc in/L\rc$ for $0 \le i < L$. 
For any $k \in [n]$ there exists some $0 \le i < L$ such that $\lc in/L \rc < k \le \lc (i+1)n/L\rc$, and the weights are nonnegative. Therefore, (\ref{eq23}) implies that 
\begin{equation}
\label{eq24}
\begin{aligned}
\widehat{G}(n, n) \le \max \limits_{0 \le i < L} \max \{&\widehat{G}(\lc (i+1)n/L\rc, 0) + G(\lf (1-i/L)n \rf, n) \circ \theta_{\lc in/L \rc, 0},\\
&\widehat{G}(0, \lc (i+1)n/L\rc) + G(n, \lf (1-i/L)n \rf) \circ \theta_{0, \lc in/L \rc} \}.
\end{aligned}
\end{equation} 
By stationarity of $\bbP$, we have the following limits in $\bbP$-probability. 
\begin{equation}
\label{E15}
\begin{aligned}
\lim_{n \rightarrow \infty} \frac{G(\lf (1-i/L)n \rf, n) \circ \theta_{\lc in/L \rc, 0}}{n} &= g\left(1-i/L, 1\right) \\ \lim_{n \rightarrow \infty} \frac{G(n, \lf (1-i/L)n \rf) \circ \theta_{0, \lc in/L \rc}}{n} &= g(1, 1-i/L)
\end{aligned}
\end{equation}
Hence, these limits are $\bbP$-a.s. and, consequently, $\bfP_{\bfa, \bfb}$ a.s. $\mu$-a.s. if $n \rightarrow \infty$ along a suitable sequence $(n_k)_{k \in \bbN}$. Also, by Lemma \ref{L4}, there is a subsequence $(n'_k)_{k \in \bbN}$ in $\bbN$ $\mu$-a.s. such that $\bfP_{\bfa, \bfb}^z$ a.s. 
\begin{equation}
\label{E14}
\begin{aligned}
\lim_{k \rightarrow \infty} \frac{\widehat{G}(\lc (i+1)n'_k/L\rc, 0)}{n'_k} = g_z(i+1/L, 0) \quad \lim_{k \rightarrow \infty} \frac{\widehat{G}(0, \lc (i+1)n'_k/L\rc)}{n'_k} = g_z(0, i+1/L)
\end{aligned}
\end{equation}
Because $\bfP_{\bfa, \bfb}$ is a projection of $\bfP_{\bfa, \bfb}^z$, we can choose $(\bfa, \bfb)$ such that (\ref{E15}) and (\ref{E14}) hold $\bfP_{\bfa, \bfb}^z$-a.s. Hence, we obtain from (\ref{eq24}) that 
\[
\begin{aligned}
g_z(1, 1) &\le \max \limits_{0 \le i < L}\max\{g_z((i+1)/L, 0) + g(1-i/L, 1), \ g_z(0, (i+1)/L) + g(1, 1-i/L)\} \\
&\le \sup \limits_{0 \le t \le 1} \max \{g(t, 1) + g_z(1-t, 0), \ g(1, t) + g_z(0, 1-t)\} + \frac{\E[(a+z)^{-1}]+\E[(b-z)^{-1}]}{L}.
\end{aligned}
\]
Finally, let $L \rightarrow \infty$. The geometric model is treated similarly. 
\end{proof}

\section{Variational characterization of the shape function}
\label{S5}

We now prove Theorems \ref{T1} and \ref{T2}. The assumption $\ubar{\alpha}+\ubar{\beta} > 0$ is in force until the proof of Theorem \ref{T2}. We begin with computing $g$ on the boundary. Recall that $g$ is extended to the boundary of $\bbR_+^2$ through limits. By homogeneity, it suffices to determine $g(1, 0)$ and $g(0, 1)$. 
\begin{lem}
\label{l7}
\[
\begin{aligned}
g(1, 0) = \E\left[\frac{1}{a+\ubar{\beta}}\right] \qquad g(0, 1) = \E\left[\frac{1}{b+\ubar{\alpha}}\right]. \\
\end{aligned}
\]
\end{lem}
\begin{proof}
We have $g(1, 0) \le g_z(1, 0) = \E[(a+z)^{-1}]$ for all $z \in (-\ubar{\alpha}, \ubar{\beta})$. Letting $z \uparrow \ubar{\beta}$ yields the upper bound $g(1, 0) \le \E[(a+\ubar{\beta})^{-1}]$. Now the lower bound. Let $\epsilon > 0$. By Lemma \ref{L1}, (\ref{E16}) and since $\mu(b_1 \le \ubar{\beta}+\epsilon) > 0$, there exists $(\bfa, \bfb)$ such that $b_1 \le \ubar{\beta}+\epsilon$ and 
\begin{align}
\lim_{n \rightarrow \infty} \frac{G(n, \lf n\epsilon \rf)}{n} &= g(1, \epsilon) \quad \bfP_{\bfa, \bfb}\text{-a.s.}\label{E17}\\
\lim_{n \rightarrow \infty} \frac{1}{n}\sum_{i=1}^{n} I(i, 0) &= \E\left[\frac{1}{a+\ubar{\beta}+\epsilon}\right]\quad \text{ in } \bfQ_{\bfa}^{\ubar{\beta}+\epsilon}\text{-probability.} \label{E18}
\end{align}
($\bfQ_{\bfa}^z$ is defined immediately after (\ref{E16})). 
The distribution of $\{W(i, 1): 1 \le i \le n\}$ under $\bfP_{\bfa, \bfb}$ stochastically dominates the distribution of $\{I(i, 0): 1 \le i \le n\}$ under $\bfQ_{\bfa}^{\ubar{\beta}+\epsilon}$ as these distributions have product forms and $i$th marginals are exponentials with rates $a_i+b_1 \le a_i+\ubar{\beta}+\epsilon$ for $i \in [n]$. Therefore, for $x \in \bbR$ and $n \ge 1/\epsilon$,  
\begin{align*}
\bfP_{\bfa, \bfb}(G(n, \lf n\epsilon \rf) \ge nx) \ge \bfP_{\bfa, \bfb}\left(\sum_{i=1}^n W(i, 1) \ge nx\right) \ge \bfQ_{\bfa}^{\ubar{\beta}+\epsilon}\left(\sum_{i=1}^n I(i, 0) \ge nx \right). 
\end{align*}
Set $x = \E[(a+\ubar{\beta}+\epsilon)^{-1}]-\epsilon$ and let $n \rightarrow \infty$. By (\ref{E17}) and (\ref{E18}), we obtain $g(1, \epsilon) \ge x$. Sending $\epsilon \downarrow 0$ gives $g(1, 0) \ge \E[(a+\ubar{\beta})^{-1}]$. Computation of $g(0, 1)$ is similar.  
\end{proof}

We now extract $g$ from (\ref{E8}). For this, we will only use the boundary values of $g$ provided in Lemma \ref{l7}, and that $A(z) = \E[(a+z)^{-1}]$ and $B(z) = \E[(b-z)^{-1}]$ are continuous, stricly monotone functions on $(-\ubar{\alpha}, \ubar{\beta})$. 
\begin{lem} 
\label{l8}
Let $r$ be a positive, continuous function on $[0, \pi/2]$. For $z \in (-\ubar{\alpha}, \ubar{\beta})$,  
\[\sup_{0 \le \theta \le \pi/2} \{g(x(\theta), y(\theta))-g_z(x(\theta), y(\theta))\} = 0, \]
where $(x(\theta), y(\theta)) = (r(\theta)  \cos \theta,  r(\theta) \sin \theta)$ for  $0 \le \theta \le \pi/2$. 
\end{lem}
\begin{proof}
We can rewrite (\ref{E8.1}) as 
\begin{align}
A(z) + B(z) &= \sup \limits_{\pi/4 \le \theta \le \pi/2} \{(1-\cot \theta) A(z) + g(\cot \theta, 1)\} \vee \sup \limits_{0 \le \theta \le \pi/4} \{(1-\tan \theta)B(z) + g(1, \tan \theta)\} \nonumber\\
&= \sup \limits_{\pi/4 \le \theta \le \pi/2} \left\{\left(1-\frac{x(\theta)}{y(\theta)}\right) A(z) + g\left(\frac{x(\theta)}{y(\theta)}, 1\right) \right\} \nonumber\\ 
&\ \vee \sup \limits_{0 \le \theta \le \pi/4} \left\{\left(1-\frac{y(\theta)}{x(\theta)}\right)B(z) + g\left(1, \frac{y(\theta)}{x(\theta)}\right)\right\}, \nonumber
\end{align}
where we use that $x$ and $y$ are nonzero, respectively, on the intervals $[0, \pi/4]$ and $[\pi/4, \pi/2]$. Collecting the terms on the right-hand side and using homogeneity, we obtain that  
\begin{equation}
\label{eq26}
\begin{aligned}
0 = \max \bigg \{&\sup \limits_{\pi/4 \le \theta \le \pi/2} \frac{1}{y(\theta)}\{-x(\theta)A(z) -y(\theta) B(z)+ g(x(\theta), y(\theta))\},\\ &\sup \limits_{0 \le \theta \le \pi/4} \frac{1}{x(\theta)}\{-x(\theta)A(z) -y(\theta) B(z)+ g(x(\theta), y(\theta))\} \bigg \}.
\end{aligned}
\end{equation}
The expressions inside the supremums in (\ref{eq26}) are continuous functions of $\theta$ over closed intervals. Hence, there exists $\theta_z \in [0, \pi/2]$ such that 
\begin{align*}
0 = -x(\theta_z)A(z) - y(\theta_z) B(z) + g(x(\theta_z), y(\theta_z)) = \sup \limits_{0 \le \theta \le \pi/2} \{-x(\theta)A(z) -y(\theta) B(z)+ g(x(\theta), y(\theta))\}, 
\end{align*}
where the second equality is due to $g \le g_z$.  
\end{proof}

\begin{cor} 
\label{c2}
\begin{equation}
\label{eq89}
B(z) = \sup \limits_{0 \le s < \infty} \{-sA(z) + g(s, 1)\} \quad \text{ for } z \in (-\ubar{\alpha}, \ubar{\beta}).  
\end{equation}
\end{cor}
\begin{proof}
Let $S > 0$. The set $\{(s, 1): 0 \le s \le S\} \cup \{(S, t): 0 \le t \le 1\}$ is the image of a curve $\theta \mapsto (r(\theta)\cos\theta, r(\theta)\sin\theta)$ for $[0, \pi/2]$ with continuous and positive $r$. Hence, by Lemma \ref{l8},  
\begin{equation}
\label{eq88}
0 = \max \{ \sup \limits_{0 \le s \le S} \{g(s, 1) - g_z(s, 1)\}, \sup \limits_{0 \le t \le 1} \{g(S, t) - g_z(S, t)\}\}.
\end{equation}
Using homogeneity and Lemma \ref{l7}, we observe that  
\begin{equation}
\label{90}
\begin{aligned}
g(S, t)-g_z(S, t) &= g(S, t) - SA(z)-tB(z) \le S(g(1, 1/S) - A(z)) \rightarrow -\infty \quad \text{ as } S \rightarrow \infty.
\end{aligned}
\end{equation}
Hence, the second supremum in (\ref{eq88}) can be dropped provided that $S$ is sufficiently large, which results in 
$
0 = \sup_{0 \le s \le S} \{g(s, 1) - g_z(s, 1)\}.
$
This equality remains valid if $S$ is replaced with $\infty$ by (\ref{90}) with $t = 1$. Rearranging terms gives (\ref{eq89}). 
\end{proof}

\begin{proof}[Proof of Theorem \ref{T1}]
Define the function $\gamma:\bbR \rightarrow \bbR \cup \{\infty\}$ by $\gamma(s) = -g(s, 1)$ for $s \ge 0$ and $\gamma(s) = \infty$ for $s < 0$. By Proposition \ref{L1}, $\gamma$ is nonincreasing, continuous and convex on $[0, \infty)$ and completely determines $g$. Let $\gamma^*$ denote the convex conjugate of $\gamma$, that is,  
\begin{align}
\gamma^*(x) &= \sup \limits_{s \in \bbR} \{sx-\gamma(s)\} = \sup \limits_{s \ge 0} \{sx - \gamma(s)\} \quad \text{ for } x \in \bbR. \label{eq30}
\end{align}
Let $f$ be the function whose graph is the image of the curve $z \mapsto (-A(z), B(z))$. That is, $f$ is defined on the interval 
$(-A(-\ubar\alpha), -A(\ubar{\beta}))$ and is given by the formula $f(x) = B \circ A^{-1}(-x)$. 
By Corollary \ref{c2}, 
\begin{equation}\label{eq29} f(x) = \sup \limits_{0 \le s < \infty} \{sx-\gamma(s)\} \quad \text{ for } x \in (-A(-\ubar{\alpha}), -A(\ubar{\beta}))\end{equation} Comparison of (\ref{eq30}) and (\ref{eq29}) shows that $\gamma^*$ coincides with $f$ on $(-A(-\ubar{\alpha}), -A(\ubar{\beta}))$. 
Since $\gamma$ is a lower semi-continuous, proper convex function on the real line, by the Fenchel-Moreau theorem, $\gamma$ equals the convex conjugate of $\gamma^*$, hence, 
\begin{equation}
\label{eq31}
\gamma(s) = \sup \limits_{x \in \bbR} \{sx - \gamma^*(x)\} \quad \text{ for } s \in \bbR
\end{equation} 

To prove the result, we need to show the supremum in (\ref{eq31}) can be taken over the interval $(-A(-\ubar{\alpha}), -A(\ubar{\beta}))$ instead of the real line. It is clear from (\ref{eq30}) that $\gamma^*$ is nondecreasing and is bounded below by $-\gamma(0) = g(0, 1) = B(-\ubar{\alpha})$. Since $\gamma^*$ agrees with $f$ on  $(-A(-\ubar{\alpha}), -A(\ubar{\beta}))$, 
\begin{equation}
\label{eq32}
\begin{aligned}
B(-\ubar{\alpha}) &\le \gamma^*(-A(-\ubar{\alpha})) \le \lim \limits_{x \downarrow -A(-\ubar{\alpha})} f(x) \\ 
&= \lim \limits_{x \downarrow -A(-\ubar{\alpha})} B \circ A^{-1}(-x) = \lim \limits_{z \rightarrow -\ubar{\alpha}} B(z)  = B(-\ubar{\alpha}), 
\end{aligned}
\end{equation}
where we used continuity of $A^{-1}$ and $B$. Hence, $\gamma^*(x) = B(-\ubar{\alpha})$ for $x \le -A(-\ubar{\alpha})$. On the other hand, if $x > -A(\ubar{\beta}) = -g(1, 0)$ then 
$\gamma^*(x) = \infty$ by (\ref{eq30}) because 
\[
\lim \limits_{s \rightarrow \infty} sx - \gamma(s) = \lim \limits_{s \rightarrow \infty} s(x + g(1, 1/s)) = \infty. 
\]
Finally, we compute $\gamma^*$ at $-A(\ubar{\beta})$. Being a convex conjugate, 
$\gamma^*$ is lower semi-continuous. Since $\gamma^*$ is also nondecreasing, 
$\lim_{y \uparrow x} \gamma^*(y) = \gamma^*(x)$ for any $x \in \bbR$. Then, proceeding as in (\ref{eq32}), 
\[\gamma^*(-A(\ubar{\beta})) = \lim \limits_{x \uparrow -A(\ubar{\beta})}f(x) = B(\ubar{\beta}).\] 
We conclude that the function $x \mapsto sx - \gamma^*(x)$ is increasing for $x \le -A(-\ubar{\alpha})$ and is $-\infty$ for $x > -A(\ubar{\beta})$. Moreover, the left- and right-hand limits agree with the value of the function at $-A(\ubar{\beta})$ and $-A(-\ubar{\alpha})$, respectively. Hence, by (\ref{eq31}), 
\[
\begin{aligned}
\gamma(s) &= \sup \limits_{s \in (-A(-\ubar{\alpha}), -A(\ubar{\beta}))} \{sx - \gamma^*(x)\} = \sup \limits_{z \in (-\ubar{\alpha}, \ubar{\beta})} \{-sA(z)-B(z)\} = -\inf \limits_{z \in (-\ubar{\alpha}, \ubar{\beta})} \{sA(z) + B(z)\}, 
\end{aligned}
\]
which implies (\ref{eq3}). 
\end{proof}

\begin{proof}[Proof of Theorem \ref{T2}]
Introduce $\delta > 0$ and let $\varphi: \bbR_+^{\bbN} \rightarrow \bbR_+^{\bbN}$ denote the map $(c_n)_{n \in \bbN} \mapsto (c_n \vee \delta)_{n \in \bbN}$. Because $\varphi$ commutes with the shift $\tau_1$, $\varphi(\bfa)$ and $\varphi(\bfb)$ are stationary sequences in $(0, \infty)$. Moreover, for each $k, l \in \bbN$, the distribution $\mu_\delta$ of $(\varphi(\bfa), \varphi(\bfb))$ is ergodic with respect to $\tau_k \times \tau_l$. To see this, suppose that $B = (\tau_k \times \tau_l)^{-1}(B)$ for some $k, l \in \bbN$ and Borel set $B \subset \bbR_+^{\bbN} \times \bbR_+^{\bbN}$. Then $(\varphi \times \varphi)^{-1}(B) = (\varphi \times \varphi)^{-1}((\tau_k \times \tau_l)^{-1}(B)) = (\tau_k \times \tau_l)^{-1}((\varphi \times \varphi)^{-1}(B))$. Hence, by the ergodicity of $\mu$, we get $\mu_\delta(B) = \mu((\varphi \times \varphi)^{-1}(B)) \in \{0, 1\}$. 
 
Let $\alpha_\delta$ and $\beta_\delta$ denote the marginal distributions of $\varphi(\bfa)$ and $\varphi(\bfb)$, respectively. Then $\ubar{{\alpha_\delta}} = \ubar{{\beta_\delta}} = \delta$. 
Applying Theorem \ref{T1} gives
\begin{align}g^{\alpha_\delta, \beta_\delta}(s, t) = \inf \limits_{z \in (-\delta, \delta)} \left\{s \E\left[\frac{1}{a \vee \delta + z}\right] + t \E\left[\frac{1}{b \vee \delta-z}\right]\right\}.\label{E9}\end{align}
Since $\bfP_{\bfa, \bfb}$ stochastically dominates $\bfP_{\varphi(\bfa), \varphi(\bfb)}$, we have $g^{\alpha_\delta, \beta_\delta}(s, t) \le g^{\alpha, \beta}(s, t)$ for $s, t \ge 0$. Using this and (\ref{E9}), we obtain 
\begin{align*}g^{\alpha, \beta}(s, t) \ge \inf \limits_{z \in (-\delta, \delta)} \left\{s \E\left[\frac{1}{a \vee \delta' + z}\right] + t \E\left[\frac{1}{b \vee \delta'-z}\right]\right\}, \end{align*}
where we fix $\delta' > \delta$. Because the expression inside the infimum is continuous in $z$, letting $\delta \downarrow 0$ yields $g^{\alpha, \beta}(s, t) \ge s\E[(a \vee \delta')^{-1}] + t\E[(b \vee \delta')^{-1}]$ for $s, t \ge 0$. 
Then, by monotone convergence, letting $\delta' \rightarrow 0$ results in
\begin{equation*}
g^{\alpha, \beta}(s, t) \ge s\E\left[\frac{1}{a}\right] + t\E\left[\frac{1}{b}\right] \quad \text{ for } s,t \ge 0. 
\end{equation*} 
The opposite inequality is noted after Lemma \ref{L4}.  
\end{proof}

\bibliographystyle{habbrv}
\bibliography{CGM_Shape}

\end{document}